\documentclass[oneside,english]{amsart}
\usepackage[T1]{fontenc}
\usepackage[latin9]{inputenc}
\usepackage{amsthm}
\usepackage{amstext}
\usepackage{amssymb}
\usepackage{esint}
\usepackage{graphicx}
\usepackage{enumerate}
\usepackage{amscd}

\makeatletter
\newtheorem{theorem}{Theorem}[section]

\newtheorem{corollary}[theorem]{Corollary}
\numberwithin{figure}{section}
\theoremstyle{definition}

\newtheorem{example}[theorem]{Example}

\theoremstyle{remark}
\newtheorem{remark}[theorem]{Remark}

\numberwithin{equation}{section}
\makeatother

\usepackage{color}

\usepackage{babel}

	\DeclareMathOperator{\loc}{loc}
	
	\DeclareMathOperator*{\esssup}{ess\,sup}

\begin{document}

\title[On Conformal Spectral Gap Estimates of the Dirichlet-Laplacian]{On Conformal Spectral Gap Estimates of the Dirichlet-Laplacian}

\author{V.~Gol'dshtein, V.~Pchelintsev, A.~Ukhlov}

\begin{abstract}
We study spectral stability estimates of the Dirichlet eigenvalues of the Laplacian in non-convex domains $\Omega\subset\mathbb R^2$. With the help of these estimates we obtain asymptotically sharp inequalities of ratios of eigenvalues in the frameworks of the Payne-P\'olya-Weinberger inequalities. These estimates are equivalent to spectral gap estimates of the Dirichlet eigenvalues of the Laplacian in non-convex domains in terms of conformal (hyperbolic) geometry. 
\end{abstract}
\maketitle

\section{Introduction}

The spectral gap problem for the two-dimensional Laplace operator 
$$
-\Delta u=-\left(\frac{\partial^2 u}{\partial x^2}+\frac{\partial^2 u}{\partial y^2}\right),\,\,(x,y)\in\Omega,
$$
arises in problems of continuum mechanics. In the present paper we obtain spectral gap estimates of the  Dirichlet eigenvalues of the Laplacian in a large class of non-convex domains $\Omega\subset\mathbb R^2$. This study is based on spectral stability estimates of the Laplace operator in so-called conformal regular domains \cite{BGU15}.

This notion of conformal regular domains was introduced in \cite{BGU15}. Recall that by Riemann's mapping theorem there exists a conformal mapping $\varphi:\mathbb D\to \Omega$ of a simply connected domain $\Omega \subset \mathbb R^2$ onto the unit disc $\mathbb D$. If $\varphi$ belongs to the Sobolev space $L^{1,\alpha}(\mathbb D)$ for some $\alpha >2$, then $\Omega$ is called a conformal $\alpha$-regular domain. Note that this definition does not depend on choice of a conformal mapping $\varphi:{\mathbb D}\to\Omega$ and can be reformulated in terms of the hyperbolic metrics \cite{BGU15}. 
It is known that any $C^2$-smooth simply connected bounded domain is $\infty$-regular (see, for example, \cite{Kr}). 

In the case of conformal $\alpha$-regular domains $\Omega,\widetilde{\Omega}\subset \mathbb R^2$ we introduce a new invariant that we call a conformal $\alpha$-variation:
\begin{equation}
\label{defvar}
V_{\alpha}^{0}(\Omega,\widetilde{\Omega})=\inf_{\varphi,\widetilde{\varphi}} \left[ \left(\|\varphi' \mid L^{\alpha}(\mathbb D)\| + \|\widetilde{\varphi}' \mid L^{\alpha}(\mathbb D)\|\right)
\|\varphi'-\widetilde{\varphi}' \mid L^{2}(\mathbb D)\| \right],
\end{equation}
where the infimum is taken over all conformal mappings $\varphi:\mathbb D\to\Omega$ and $\widetilde{\varphi}:\mathbb D\to\widetilde{\Omega}$.

By this definition $V_{\alpha}^{0}(\Omega,\widetilde{\Omega}) \to 0$ if $\|\varphi'-\widetilde{\varphi}' \mid L^{2}(\mathbb D)\|  \to 0$. In this sense it is an asymptotic invariant.

{\it This invariant was used in spectral stability estimates \cite{BGU15} but was not extracted from the right hand side of these estimates.}

The conformal $\alpha$-variation measures a "distance" between $\Omega$ and $\widetilde{\Omega}$ in terms of $L^2$-norms of conformal homeomorphisms. Using a notion of the conformal radius it can be proved that it depends on hyperbolic metrics only. In this paper the "conformal" version is more convenient. 
If $\Omega=\widetilde{\Omega}$ then $V_{\alpha}^{0}(\Omega,\widetilde{\Omega})=0$,  but in the case of different $\Omega$ and $\widetilde{\Omega}$ existence of the extremal conformal mappings is an open problem.

The suggested method is based on the spectral stability estimates of the Dirichlet-Laplace operator \cite{BGU15},  on the geometric theory of composition operators on Sobolev spaces \cite{U93,VU02} and its applications to the Sobolev type embedding theorems \cite{GG94,GU09}. 

Let $\Omega\subset\mathbb R^2$ be an open set. The Sobolev space $W^{1,p}(\Omega)$, $1\leq p<\infty$, (see, for example, \cite{M}) is defined 
as a Banach space of locally integrable weakly differentiable functions
$f:\Omega\to\mathbb{R}$ endowed with the following norm: 
\[
\|f\mid W^{1,p}(\Omega)\|=\biggr(\iint\limits _{\Omega}|f(x,y)|^{p}\, dxdy\biggr)^{\frac{1}{p}}+
\biggr(\iint\limits _{\Omega}|\nabla f(x,y)|^{p}\, dxdy\biggr)^{\frac{1}{p}}.
\]

The seminormed Sobolev space $L^{1,p}(\Omega)$, $1\leq p<\infty$, (see, for example, \cite{M}) is defined 
as a space of locally integrable weakly differentiable functions
$f:\Omega\to\mathbb{R}$ endowed with the following seminorm: 
\[
\|f\mid L^{1,p}(\Omega)\|=
\biggr(\iint\limits _{\Omega}|\nabla f(x,y)|^{p}\, dxdy\biggr)^{\frac{1}{p}}.
\]

The Sobolev space $W^{1,p}_{0}(\Omega)$, $1 \leq p< \infty$, is the closure in the $W^{1,p}(\Omega)$-norm of the 
space $C^{\infty}_{0}(\Omega)$ of all infinitely continuously differentiable functions with compact support in $\Omega$.

The proposed method permit us to obtain estimates of the spectral gap of the two-dimen\-si\-o\-nal Laplace operator in the terms of the conformal geometry of domains. Firstly we obtain estimates of the Poincar\'e constant in the Poincar\'e-Sobolev inequality for the critical case $p=n=2$. As an application we obtain the inverse Payne-P\'olya-Weinberger inequality in conformal regular domains.

These estimates can be precised for Ahlfors-type domains (i.e. quasidiscs) in terms of quasiconformal characteristics of domains.
Recall that $K$-quasidiscs are images of the unit discs under $K$-quasicon\-for\-mal homeomorphisms of the plane $\mathbb R^2$. 
The class of quasidiscs includes all Lipschitz simply connected domains but also includes some of fractal domains (for example, von Koch snowflake \cite{G82}, Rohde snowflakes \cite{Roh}). 
The Hausdorff dimension of the quasidisc's boundary can be any number in $[1,2)$.

\section{Estimates of eigenvalues of the Dirichlet-Laplacian}

Let $\Omega\subset \mathbb R^n$ be a bounded domain. Dirichlet eigenvalues $\lambda_k = \lambda_k(\Omega)$ of the Laplace operator are solutions of the following problem:
\begin{equation}\label{DirLap}
\Delta u+\lambda_k u =0\,\,\, \text{in $\Omega$}, \quad u=0\,\,\, \text{on $\partial \Omega$}.
\end{equation}

Payne, P\'olya and Weinberger \cite{PPW55, PPW56} proved that the ratio $\lambda_2(\Omega)/\lambda_1(\Omega)\leq 3$ for planar domains and conjectured that the ratio of the first two eigenvalues of the 
Dirichlet-Laplace operator in $\Omega \subset \mathbb R^2$ obtains the maximal upper bound in the disc $\mathbb D$: 
\begin{equation}\label{IneqPPW}
\frac{\lambda_2(\Omega)}{\lambda_1(\Omega)} \leq \frac{\lambda_2(\mathbb D)}{\lambda_1(\mathbb D)}.
\end{equation}

This upper bound $3$ was improved by Brands \cite{Br}, de Vries \cite{Vr} and Chiti \cite{Ch}. In \cite{AB91, AB92} Ashbaugh and Benguria proved this inequality \eqref{IneqPPW} in the case of space domains $\Omega\subset\mathbb R^n$, $n\geq 2$.

In the present section we discuss lower bounds of the ratio ${\lambda_2(\Omega)}/{\lambda_1(\Omega)}$, $\Omega\subset\mathbb R^2$. The main result gives an asymptotic sharp lower bound in the case of conformal regular domains $\Omega\subset\mathbb R^2$. As usual, lower bounds of this ratio give lower estimates of the spectral gap of the Dirichlet-Laplace operator that are important in problems of continuum mechanics.

\subsection{Refinement of the Dirichlet spectral stability estimates}

Now we refine the spectral stability theorem for conformal regular domains \cite{BGU15}.
Firstly we estimate the constant in the Poincar\'e--Sobolev inequality that appears in the spectral estimate \cite{BGU15}:

\begin{theorem}
\label{PoinConst}
Let $f \in W^{1,2}_0(\mathbb D)$. Then 
\begin{equation}\label{InPS}
\|f \mid L^{r}(\mathbb D)\| \leq A_{r,2}(\mathbb D) \|\nabla f \mid L^{2}(\mathbb D)\|, \,\,r \geq 2,
\end{equation}
where
\[
A_{r,2}(\mathbb D) \leq \inf\limits_{p\in \left(\frac{2r}{r+2},2\right)} 
\left(\frac{p-1}{2-p}\right)^{\frac{p-1}{p}}
\frac{\pi^{\frac{2-r}{2r}} 2^{-\frac{1}{p}}}{\sqrt{\Gamma(2/p) \Gamma(3-2/p)}}.
\]
\end{theorem} 

\begin{proof} 
We estimate the constant $A_{r,2}^2(\mathbb D)$ using the Talenti estimate \cite{Tal76}
\[
\|f \mid L^q(\mathbb R^n)\|\leq A_{p,q}(\mathbb R^n) \|\nabla f \mid L^p(\mathbb R^n)\|,\,\,q=\frac{np}{n-p},
\]
where 
\begin{equation*}\label{EsTal}
A_{p,q}(\mathbb R^n)= \frac{1}{\sqrt{\pi}\cdot \sqrt[p]{n}} \left(\frac{p-1}{n-p}\right)^{\frac{p-1}{p}}
\left(\frac{\Gamma(1+n/2) \Gamma(n)}{\Gamma(n/p) \Gamma(1+n-n/p)}\right)^{\frac{1}{n}}.
\end{equation*}

The Talenti estimate can not be applied directly for $p=n=2$. Choose some number $p: 2r/(2+r)<p<2$. By the H\"older inequality with exponents $(2/(2-p), 2/p)$ we have
\begin{multline*}
\biggr(\iint\limits _{\mathbb D}|\nabla f(x,y)|^{p}\, dxdy\biggr)^{\frac{1}{p}} \leq
\biggr(\iint\limits _{\mathbb D} dxdy\biggr)^{\frac{2-p}{2p}} \biggr(\iint\limits _{\mathbb D}|\nabla f(x,y)|^{2}\, dxdy\biggr)^{\frac{1}{2}}\\ = 
\pi^{\frac{2-p}{2p}}\biggr(\iint\limits _{\mathbb D}|\nabla f(x,y)|^{2}\, dxdy\biggr)^{\frac{1}{2}}.
\end{multline*} 
Because any function $f\in W^{1,p}_0(\mathbb D)$ can be extended by zero to $\widetilde{f}\in W^{1,p}_0(\mathbb R^n)$, it permit us to apply the Talenti estimate:
\begin{multline*}
\biggr(\iint\limits _{\mathbb D}|f(x,y)|^{q}\, dxdy\biggr)^{\frac{1}{q}}=\biggr(\iint\limits _{\mathbb R^2}|\widetilde{f}(x,y)|^{q}\, dxdy\biggr)^{\frac{1}{q}}\\
 \leq A_{p,q}(\mathbb R^2)\biggr(\iint\limits _{\mathbb R^2}|\nabla \widetilde{f}(x,y)|^{p}\, dxdy\biggr)^{\frac{1}{p}}=
A_{p,q}(\mathbb R^2)\biggr(\iint\limits _{\mathbb D}|\nabla f(x,y)|^{p}\, dxdy\biggr)^{\frac{1}{p}},
\end{multline*}
where
\[
A_{p,q}(\mathbb R^2) =  
\frac{1}{\sqrt{\pi}\cdot \sqrt[p]{2}} \left(\frac{p-1}{2-p}\right)^{\frac{p-1}{p}}
\frac{1}{\sqrt{\Gamma(2/p) \Gamma(3-2/p)}}.
\]

Taking into account the H\"older inequality with exponents $(q/(q-r), q/r)$ we get
\begin{multline*}
\biggr(\iint\limits _{\mathbb D}|f(x,y)|^{r}\, dxdy\biggr)^{\frac{1}{r}} \leq 
\biggr(\iint\limits _{\mathbb D} dxdy\biggr)^{\frac{q-r}{qr}} \biggr(\iint\limits _{\mathbb D}|f(x,y)|^{q}\, dxdy\biggr)^{\frac{1}{q}} \\=
\pi^{\frac{q-r}{qr}}\biggr(\iint\limits _{\mathbb D}|f(x,y)|^{q}\, dxdy\biggr)^{\frac{1}{q}}
\leq\pi^{\frac{q-r}{qr}}A_{p,q}(\mathbb R^2)
\biggr(\iint\limits _{\mathbb D}|\nabla f(x,y)|^{p}\, dxdy\biggr)^{\frac{1}{p}}
\\
\leq \pi^{\frac{q-r}{qr}}\pi^{\frac{2-p}{2p}}A_{p,q}(\mathbb R^2)
\biggr(\iint\limits _{\mathbb D}|\nabla f(x,y)|^{2}\, dxdy\biggr)^{\frac{1}{2}}.
\end{multline*}

Since the last inequality holds for any $p\in (2r/(2+r), 2)$ and $q=2p/(2-p)$ we obtain that
\[
\biggr(\iint\limits _{\mathbb D}|f(x,y)|^{r}\, dxdy\biggr)^{\frac{1}{r}} \leq A_{r,2}(\mathbb D)
\biggr(\iint\limits _{\mathbb D}|\nabla f(x,y)|^{2}\, dxdy\biggr)^{\frac{1}{2}},
\]
where
\[
A_{r,2}(\mathbb D) \leq \inf\limits_{p\in \left(\frac{2r}{r+2},2\right)} 
\left(\frac{p-1}{2-p}\right)^{\frac{p-1}{p}}
\frac{\pi^{\frac{2-r}{2r}} 2^{-\frac{1}{p}}}{\sqrt{\Gamma(2/p) \Gamma(3-2/p)}}.
\]

\end{proof}

This estimate of the constant $A_{r,2}$ allows us to refine spectral stability estimates of the work \cite{BGU15}.

\begin{theorem}\label{BGU}
Let $\Omega$ and $\widetilde{\Omega}$ be conformal $\alpha$-regular domains for some $\alpha \in (2, \infty]$.
Then for any $k \in \mathbb N$
\begin{equation*}\label{InBGU}
|\lambda_k(\Omega)-\lambda_k(\widetilde{\Omega})|
\leq \max \left\{\lambda_k^2(\Omega),\lambda_k^2(\widetilde{\Omega})\right\} A_{r,2}^2(\mathbb D)V_{\alpha}^{0}(\Omega,\widetilde{\Omega}),
\end{equation*}
where 
$$
A_{r,2}^2(\mathbb D) \leq \gamma_{\alpha} = \inf\limits_{p\in \left(\frac{4 \alpha}{3\alpha -2},2\right)} 
\left(\frac{p-1}{2-p}\right)^{\frac{2(p-1)}{p}}
\frac{\pi^{-\frac{\alpha +2}{2\alpha}} 4^{-\frac{1}{p}}}{\Gamma(2/p) \Gamma(3-2/p)}, \quad  r=\frac{4 \alpha}{\alpha -2}
$$
is the exact constant in inequality \eqref{InPS}.
\end{theorem}

\begin{proof} 
Because $\Omega$ and $\widetilde{\Omega}$ are conformal $\alpha$-regular domains there exist conformal mappings $\varphi: \mathbb D \to \Omega$ and $\widetilde{\varphi}: \mathbb D \to \widetilde{\Omega}$ such that $|\varphi'|, |\widetilde{\varphi}'| \in L^{\alpha}(\mathbb D)$.
In \cite{BGU15} was proved that
\begin{multline*}
|\lambda_k(\Omega)-\lambda_k(\widetilde{\Omega})|\\
\leq 
\max \left\{\lambda_k^2(\Omega),\lambda_k^2(\widetilde{\Omega})\right\} A_{r,2}^2(\mathbb D) \left(\|\varphi' \mid L^{\alpha}(\mathbb D)\| + \|\widetilde{\varphi}' \mid L^{\alpha}(\mathbb D)\|\right)
\||\varphi'|-|\widetilde{\varphi}'| \mid L^{2}(\mathbb D)\|.
\end{multline*}
Since this inequality holds for any $\varphi: \mathbb D \to \Omega$, $\widetilde{\varphi}: \mathbb D \to \widetilde{\Omega}$ we obtain
\begin{multline*}\label{InBGU2}
|\lambda_k(\Omega)-\lambda_k(\widetilde{\Omega})|\leq 
\max \left\{\lambda_k^2(\Omega),\lambda_k^2(\widetilde{\Omega})\right\} A_{r,2}^2(\mathbb D)
\times \\ \times
\inf\limits_{\varphi\in L^{1,2}(\mathbb D),\widetilde{\varphi}\in L^{1,2}(\mathbb D)}
\left(\|\varphi' \mid L^{\alpha}(\mathbb D)\| + \|\widetilde{\varphi}' \mid L^{\alpha}(\mathbb D)\|\right)
\||\varphi'|-|\widetilde{\varphi}'| \mid L^{2}(\mathbb D)\|\\
= \max \left\{\lambda_k^2(\Omega),\lambda_k^2(\widetilde{\Omega})\right\} A_{r,2}^2(\mathbb D)V_{\alpha}^{0}(\Omega,\widetilde{\Omega}).
\end{multline*}

Now using Theorem~\ref{PoinConst} for estimate of the constant $A_{r,2}(\mathbb D)$ and taking $r=4 \alpha /(\alpha -2)$ we have 
$$
A_{r,2}^2(\mathbb D) \leq \gamma_{\alpha} = \inf\limits_{p\in \left(\frac{4 \alpha}{3\alpha -2},2\right)} 
\left(\frac{p-1}{2-p}\right)^{\frac{2(p-1)}{p}}
\frac{\pi^{-\frac{\alpha +2}{2\alpha}} 4^{-\frac{1}{p}}}{\Gamma(2/p) \Gamma(3-2/p)}, \quad  r=\frac{4 \alpha}{\alpha -2}.
$$

\end{proof}

Combining Theorem \ref{BGU} and some classical results of the spectral theory of elliptic operators we get upper estimates for the first eigenvalue and 
lower estimates for the second eigenvalue of the Dirichlet-Laplace operator in conformal regular domains:
\begin{theorem}\label{UDE}
Let $\Omega \subset \mathbb R^2$ be a conformal $\alpha$-regular domain of area $\pi$.
Then the following inequalities hold
\begin{equation*}\label{UpEst}
\lambda_1(\Omega) \leq 
\lambda_1(\mathbb{D})+\lambda_1^2(\mathbb D_{\rho}) A_{r,2}^2(\mathbb D) V_{\alpha}^{0}(\mathbb D,\Omega),
\end{equation*}
\begin{equation*}\label{LowEst}
\lambda_2(\Omega) \geq
\lambda_2(\mathbb D) - \lambda_{*}^2 \cdot \lambda_1^2(\mathbb D_{\rho}) A_{r,2}^2(\mathbb D) V_{\alpha}^{0}(\mathbb D,\Omega),
\end{equation*}
where $\lambda_{*}=\frac{\lambda_2(\mathbb D)}{\lambda_1(\mathbb D)} \approx 2.539$ and $\mathbb D_{\rho}$ is the largest disc inscribed in $\Omega$.
\end{theorem}

\begin{proof}
Since $\Omega$ is a conformal $\alpha$-regular domain, by Theorem \ref{BGU} in case $k=1, 2$ and $\widetilde{\Omega}=\mathbb D$, we have the following estimates:
\begin{equation}\label{Eq1}
\lambda_1(\Omega) \leq \lambda_1(\mathbb D) 
+\max \left\{\lambda_1^2(\Omega), \lambda_1^2(\mathbb D)\right\} A_{r,2}^2(\mathbb D) V_{\alpha}^{0}(\mathbb D,\Omega).
\end{equation}
\begin{equation}\label{Equ1}
\lambda_2(\Omega) \geq \lambda_2(\mathbb D)-\max \left\{\lambda_2^2(\Omega), \lambda_2^2(\mathbb D)\right\}
A_{r,2}^2(\mathbb D) V_{\alpha}^{0}(\mathbb D,\Omega).
\end{equation}

Further we indicate maximum between $\lambda_1(\Omega)$ and $\lambda_1(\mathbb D)$ as well as between $\lambda_2(\Omega)$ and $\lambda_2(\mathbb D)$.

According to the Rayleigh-Faber-Krahn inequality \cite{F23, Kh25}, which states that the disc minimizes the first
Dirichlet eigenvalue among all planar domains of the same area, i.e.
\[
\lambda_1(\Omega) \geq \lambda_1(\mathbb D) = j_{0,1}^2,
\]
we obtain
\begin{equation*}
\max \left\{\lambda_1^2(\Omega), \lambda_1^2(\mathbb D)\right\} = \lambda_1^2(\Omega).
\end{equation*}
Here $j_{0,1} \approx 2.4048$ is the first positive zero of the Bessel function $J_0$.

In turn, the property of domain monotonicity for the Dirichlet eigenvalues (see, for example, \cite{GN13}) implies the following estimate
\begin{equation}\label{Eq2}
\lambda_1(\Omega) \leq \lambda_1(\mathbb D_{\rho}) = \frac{j_{0,1}^2}{\rho^2}.
\end{equation}
where $\mathbb D_{\rho}$ is the largest ball inscribed in $\Omega$ and $\rho$ is its radius.

Now we determine maximum between $\lambda_2(\Omega)$ and $\lambda_2(\mathbb D)$. For this aim we use the Payne-P\'olya-Weinberger inequality
\[
\frac{\lambda_2(\Omega)}{\lambda_1(\Omega)} \leq \frac{\lambda_2(\mathbb D)}{\lambda_1(\mathbb D)} = \lambda_{*} \approx 2.539.
\] 
Using that $\lambda_1(\Omega)/\lambda_1(\mathbb D) \geq 1$ for $|\Omega|=\pi$ and taking into account estimate \eqref{Eq2},  
straightforward calculations yield 
\begin{multline}\label{Equ2}
\max \left\{\lambda_2^2(\Omega), \lambda_2^2(\mathbb D)\right\} 
\leq \max \left\{\frac{\lambda_2^2(\mathbb {D})}{\lambda_1^2(\mathbb {D})} \lambda_1^2(\Omega), \lambda_2^2(\mathbb D)\right\} \\
= \frac{\lambda_2^2(\mathbb {D})}{\lambda_1^2(\mathbb {D})} \lambda_1^2(\Omega) 
\leq \frac{\lambda_2^2(\mathbb {D})}{\lambda_1^2(\mathbb {D})} \lambda_1^2(\mathbb D_{\rho})
= \lambda_{*}^2 \cdot \lambda_1^2(\mathbb D_{\rho}).
\end{multline}

Finally, combining inequalities \eqref{Eq2}, \eqref{Eq1} and   
inequalities \eqref{Equ2}, \eqref{Equ1} we obtain the required result. 
\end{proof}

\begin{remark}
In Theorem \ref{UDE}, instead of the suggested upper estimate for the first Dirichlet eigenvalue, one can use  
the well-known upper estimate for the first eigenvalue of the Dirichlet-Laplacian 
in simply connected planar domains received by Payne and Weinberger \cite{PW}: 
\[
\lambda_1(\Omega) \leq \frac{\pi j^2_{0,1}}{|\Omega|} \left[1+\left(\frac{1}{J_1^2(j_{0,1})}-1\right) \left(\frac{|\partial \Omega|^2}{4 \pi |\Omega|}-1\right)\right],
\]
where $|\Omega|$ is the Lebesgue measure of $\Omega$, $|\partial \Omega|$ is the Hausdorff measure of the boundary of $\Omega$ and  
$J_1$ denotes the Bessel function of the first kind of order one with $J_1(j_{0,1})$. 
This assertion is optimal in the sense that the equality holds if and only if $\Omega$ is a disc.
\end{remark}

For example, if $\Omega$ is bounded by von Koch snowflake, then it is known that $|\partial \Omega| =\infty$ and $|\Omega|<\infty$.
In this case, the upper estimates of Payne-Weinberger for the first of the Dirichlet eigenvalue tends to infinity.
Conversely, in \cite{GPU17_2} was shown that von Koch snowflake is a conformal $\alpha$-regular domain.

\subsection{Inverse Payne-P\'olya-Weinberger Inequality}

In \cite{PPW55} Payne, P\'olya and Weinberger studied estimates of ratios of Dirichlet eigenvalues.
In two-dimensional case, they proved that the ratio $\lambda_2(\Omega)/\lambda_1(\Omega)$ is bounded by $3$ and  
conjectured that among all planar Euclidean domains, the disc
maximizes the ratio $\lambda_2(\Omega)/\lambda_1(\Omega)$ of first and second Dirichlet eigenvalues.
This conjecture was proved by Ashbaugh and Benguria \cite{AB91}. In \cite{AB92} they
 also extended the their results to the higher-dimensional case and  
to the hemisphere in $S^n$ \cite{AB01}. 

In this paper we give asymptotically exact lower estimates of the Payne-P\'olya-Weinberger ratio of eigenvalues of the Dirichlet-Laplace operator
in conformal $\alpha$-regular domains using the asymptotic invariant $V_{\alpha}^{0}(\Omega,\widetilde{\Omega})$.  
Recall that $V_{\alpha}^{0}(\mathbb D,\Omega) \to 0$ if $\| 1-\varphi' \mid L^{2}(\mathbb D)\| \to 0$.

Thus, taking into account Theorem \ref{UDE} and performing straightforward calculations we obtain the main result of the paper.

\begin{theorem}
\label{PPW}
Let $\Omega \subset \mathbb R^2$ be a conformal $\alpha$-regular domain of area $\pi$.
Then the ratio of the first two eigenvalues of the Dirichlet-Laplacian satisfies 
\[
\frac{\lambda_2(\Omega)}{\lambda_1(\Omega)} \geq
\frac{\lambda_2(\mathbb D) - \lambda_{*}^2 \lambda_1^2(\mathbb D_{\rho}) \gamma_{\alpha} V_{\alpha}^{0}(\mathbb D,\Omega)}
{\lambda_1(\mathbb{D})+\lambda_1^2(\mathbb D_{\rho}) \gamma_{\alpha} V_{\alpha}^{0}(\mathbb D,\Omega)},
\]
where $\lambda_{*}=\frac{\lambda_2(\mathbb D)}{\lambda_1(\mathbb D)} \approx 2.539$, $\mathbb D_{\rho}$ is the largest disc inscribed in $\Omega$ and 
\[ 
\gamma_{\alpha} = \inf\limits_{p\in \left(\frac{4 \alpha}{3\alpha -2},2\right)} 
\left(\frac{p-1}{2-p}\right)^{\frac{2(p-1)}{p}}
\frac{\pi^{-\frac{\alpha +2}{2\alpha}} 4^{-\frac{1}{p}}}{\Gamma(2/p) \Gamma(3-2/p)}\,.
\]
\end{theorem}

In the case of conformal $\infty$-regular domains, we have the following assertion:
\begin{corollary}\label{Cor1}
Let $\Omega \subset \mathbb R^2$ be a conformal $\infty$-regular domain of the area $\pi$.
Then the ratio of the first two eigenvalues of the Dirichlet-Laplacian satisfies 
\begin{equation}\label{LB}
\frac{\lambda_2(\Omega)}{\lambda_1(\Omega)} \geq
\frac{\lambda_2(\mathbb D) -  \lambda_{*}^2 \lambda_1^2(\mathbb D_{\rho})\gamma_{\infty} V_{\infty}(\mathbb D,\Omega)}
{\lambda_1(\mathbb{D})+\lambda_1^2(\mathbb D_{\rho})\gamma_{\infty} V_{\infty}(\mathbb D,\Omega)},
\end{equation}
where $\lambda_{*}=\frac{\lambda_2(\mathbb D)}{\lambda_1(\mathbb D)} \approx 2.539$, $\mathbb D_{\rho}$ is the largest disc inscribed in $\Omega$ and
$$
\gamma_{\infty} = \inf\limits_{p\in \left(\frac{4}{3},2\right)} 
\left(\frac{p-1}{2-p}\right)^{\frac{2(p-1)}{p}}
\frac{\pi^{-\frac{1}{2}} 4^{-\frac{1}{p}}}{\Gamma(2/p) \Gamma(3-2/p)} < \frac{1}{5},
$$

$$
V_{\infty}(\mathbb D,\Omega)= \inf\limits_{\varphi:\mathbb D \to \Omega} \left[\bigl(\|\varphi' \mid L^{\infty}(\mathbb D)\| + 1 \bigr) \|\varphi'-1 \mid L^{2}(\mathbb D)\|\right].
$$  
\end{corollary}

As an example consider domains bounded by epicycloids.

\begin{example}
For $n \in \mathbb{N}$, the diffeomorphism 
$$
\varphi(z)=\sqrt{\frac{n}{n+1}}\left(z+\frac{1}{n}z^n\right), \quad z=x+iy,
$$
is conformal and maps the unit disc $\mathbb D$ onto the domain $\Omega_n$ bounded by an epicycloid of $(n-1)$ cusps with area $\pi$.

Now we estimate the norm of the complex derivative $\varphi'$ in $L^{\infty}(\mathbb D)$ and calculate the norm of the quantity $(\varphi'-1)$ in $L^2(\mathbb D)$. 
Straightforward calculations yield
$$
\|\varphi'\,|\,L^{\infty}(\mathbb D)\|=\esssup\limits_{|z|\leq 1} \left(\left|\sqrt{\frac{n}{n+1}}\left(1+z^{n-1}\right)\right|\right)\leq \sqrt{\frac{4n}{n+1}}\,,
$$ 
\begin{multline*}
\|\varphi'-1\,|\,L^{2}(\mathbb D)\| \\
= \sqrt{\frac{n}{n+1}} \biggr(\iint\limits _{\mathbb D}\left|z^{n-1}-\frac{\sqrt{n+1}-\sqrt{n}}{\sqrt{n}}\right|^2\, dxdy\biggr)^{\frac{1}{2}} 
=\sqrt{2\pi \left(1- \sqrt{\frac{n}{n+1}}\right)}\,.
\end{multline*}
Then by Corollary \ref{Cor1} we have
\[
\frac{\lambda_2(\Omega_n)}{\lambda_1(\Omega_n)} \geq
\frac{\lambda_2(\mathbb D) - (2.539)^2 C(n)}
{\lambda_1(\mathbb{D})+ C(n)}\,,
\]
where 
$$
C(n)=\lambda_1^2(\mathbb D_{\rho}) \gamma_{\infty} \left(\sqrt{\frac{4n}{n+1}}+1\right) \sqrt{2\pi \left(1- \sqrt{\frac{n}{n+1}}\right)}, 
\quad \rho=\left(\frac{n-1}{n+1}\right)^{\frac{3}{4}}.
$$
\end{example}

Note that Theorem \ref{UDE} also allow us obtain asymptotically exact lower estimates for the spectral gap
between the first two Dirichlet eigenvalues for conformal $\alpha$-regular domains. Namely  

\vskip 0.2cm
\begin{theorem}
Let $\Omega \subset \mathbb R^2$ be a conformal $\alpha$-regular domain of area $\pi$.
Then the spectral gap for Dirichlet-Laplacian satisfies 
\[
\lambda_2(\Omega)-\lambda_1(\Omega) \geq \lambda_2(\mathbb D)-\lambda_1(\mathbb{D})-(\lambda_{*}^2 +1)\lambda_1^2(\mathbb D_{\rho}) \gamma_{\alpha}  V_{\alpha}^{0}(\mathbb D,\Omega),
\]
where $\lambda_{*} \approx 2.539$, $\mathbb D_{\rho}$ is the largest disc inscribed in $\Omega$ and  
\[ 
\gamma_{\alpha} = \inf\limits_{p\in \left(\frac{4 \alpha}{3\alpha -2},2\right)} 
\left(\frac{p-1}{2-p}\right)^{\frac{2(p-1)}{p}}
\frac{\pi^{-\frac{\alpha +2}{2\alpha}} 4^{-\frac{1}{p}}}{\Gamma(2/p) \Gamma(3-2/p)}\,.
\]
\end{theorem}

Lower bounds of the spectral gap for the Dirichlet-Laplacian in terms of the geometry of $\Omega$ 
represent an important problem in mathematics with applications to continuum mechanics. 
For more details on the spectral gap, see for example \cite{A06}.

\section{Inverse Payne-P\'olya-Weinberger Conjecture for quasidiscs}

In this section we precise Theorem~\ref{PPW} for Ahlfors-type domains (i.e. quasidiscs) using 
integral estimates of conformal derivatives from \cite{GPU17_2}.

Following \cite{Ahl66} a homeomorphism $\varphi:\Omega \rightarrow \Omega'$
between planar domains is called $K$-quasiconformal if it preserves
orientation, belongs to the Sobolev class $W_{\loc}^{1,2}(\Omega)$
and its directional derivatives $\partial_{\xi}$ satisfy the distortion inequality
$$
\max\limits_{\xi}|\partial_{\xi}\varphi|\leq K\min_{\xi}|\partial_{\xi}\varphi|\,\,\,
\text{a.e. in}\,\,\, \Omega \,.
$$

For any planar $K$-quasiconformal homeomorphism $\varphi:\Omega\rightarrow \Omega'$
the following sharp result is known: $J(z,\varphi)\in L^p_{\loc}(\Omega)$
for any $1 \leq p<\frac{K}{K-1}$ (\cite{A94,G81}). Hence for any conformal mapping $\varphi:\mathbb{D}\to\Omega$
of the unit disc $\mathbb{D}$ onto a $K$-quasidisc $\Omega$ its derivatives $\varphi'\in L^p(\mathbb{D})$ for any 
$1\le p<\frac{2K^{2}}{K^2-1}$~\cite{BGU15, GU16}.

Using integrability of conformal derivatives on the base of the weak inverse H\"older inequality and the measure doubling condition \cite{GPU17_2} we obtain an estimate of the constant in the inverse H\"older inequality for Jacobians of quasiconformal mappings. The following theorem was proved but not formulated in \cite{GPU17_2}.

\vskip 0.2cm

\begin{theorem}
\label{thm:IHIN}
Let $\varphi:\mathbb R^2 \to \mathbb R^2$ be a $K$-quasiconformal mapping. Then for every disc $\mathbb D \subset \mathbb R^2$ and
for any $1<\kappa<\frac{K}{K-1}$ the inverse H\"older inequality
\begin{equation*}\label{RHJ}
\left(\iint\limits_{\mathbb D} |J_{\varphi}(x,y)|^{\kappa}~dxdy \right)^{\frac{1}{\kappa}}
\leq \frac{C_\kappa^2 K \pi^{\frac{1}{\kappa}-1}}{4}
\exp\left\{{\frac{K \pi^2(2+ \pi^2)^2}{2\log3}}\right\}\iint\limits_{\mathbb D} |J_{\varphi}(x,y)|~dxdy
\end{equation*}
holds. Here
$$
C_\kappa=\frac{10^{6}}{[(2\kappa -1)(1- \nu)]^{1/2\kappa}}, \quad \nu = 10^{8 \kappa}\frac{2\kappa -2}{2\kappa -1}(24\pi^2K)^{2\kappa}<1.
$$
\end{theorem}

If $\Omega$ is a $K$-quasidisc, then a conformal mapping $\varphi: \mathbb D\to\Omega$ allows $K^2$-quasiconformal reflection \cite{Ahl}. 
Hence, by Theorem~\ref{thm:IHIN} we obtain the following integral estimates of complex derivatives of conformal 
mapping $\varphi:\mathbb D\to\Omega$ of the unit disc onto a $K$-quasidisc $\Omega$:

\begin{corollary}\label{Est_Der}
Let $\Omega\subset\mathbb R^2$ be a $K$-quasidisc and $\varphi:\mathbb D\to\Omega$ be a conformal mapping. Suppose that  $2<\gamma<\frac{2K^2}{K^2-1}$.
Then 
\begin{equation*}\label{Ineq_2}
\left(\iint\limits_{\mathbb D} |\varphi'(x,y)|^{\gamma}~dxdy \right)^{\frac{1}{\gamma}}
\leq \frac{C_\gamma K \pi^{\frac{2-\gamma}{2 \gamma}}}{2}
\exp\left\{{\frac{K^2 \pi^2(2+ \pi^2)^2}{4\log3}}\right\}\cdot |\Omega|^{\frac{1}{2}}.
\end{equation*}
where
$$
C_\gamma=\frac{10^{6}}{[(\gamma -1)(1- \nu)]^{1/\gamma}}, \quad \nu = 10^{4 \gamma}\frac{\gamma -2}{\gamma -1}(24\pi^2K^2)^{\gamma}<1.
$$
\end{corollary}

Combining Theorem~\ref{UDE} and Corollary~\ref{Est_Der} we obtain spectral estimates of the Laplace operator with the Dirichlet boundary condition:

\vskip 0.2cm
\begin{theorem}\label{Dir_Est}
Let $\Omega \subset \mathbb R^2$ be a $K$-quasidisc of area $\pi$.
Then the following inequalities hold
\begin{equation*}
\lambda_1(\Omega) \leq 
\lambda_1(\mathbb{D})+\lambda_1^2(\mathbb D_{\rho}) M_{\alpha}(K) \|\varphi'-1 \mid L^{2}(\mathbb D)\|,
\end{equation*}
\begin{equation*}
\lambda_2(\Omega) \geq
\lambda_2(\mathbb D) - \lambda_{*}^2 \cdot \lambda_1^2(\mathbb D_{\rho}) M_{\alpha}(K) \|\varphi'-1 \mid L^{2}(\mathbb D)\|,
\end{equation*}
where $\lambda_{*}=\frac{\lambda_2(\mathbb D)}{\lambda_1(\mathbb D)} \approx 2.539$, $\mathbb D_{\rho}$ is the largest disc inscribed in $\Omega$ and 
\begin{multline*}
M_{\alpha}(K)= \inf\limits_{2< \alpha < \alpha^*} 
\Biggl\{
\inf\limits_{p\in \left(\frac{4 \alpha}{3\alpha -2},2\right)} 
\left(\frac{p-1}{2-p}\right)^{\frac{2(p-1)}{p}}
\frac{\pi^{-\frac{\alpha +2}{2\alpha}} 4^{-\frac{1}{p}}}{\Gamma(2/p) \Gamma(3-2/p)} \\
\times
\left(\frac{C_\alpha K \pi^{\frac{2-\alpha}{2\alpha}}}{2}
\exp\left\{{\frac{K^2 \pi^2(2+ \pi^2)^2}{4\log3}}\right\} \cdot |\Omega|^{\frac{1}{2}} +\pi^{\frac{1}{\alpha}} \right) \Biggr\}, \\
C_\alpha=\frac{10^{6}}{[(\alpha -1)(1- \nu(\alpha))]^{1/\alpha}}.
\end{multline*}
\end{theorem}

\begin{proof}
Quasidiscs are conformal $\alpha$-regular domains for $2<\alpha<\frac{2K^2}{K^2-1}$ \cite{GU16}. Then by Corollary~\ref{UDE} for any $2<\alpha<\frac{2K^2}{K^2-1}$ we have
\begin{equation}\label{Est_1}
\lambda_1(\Omega) \leq 
\lambda_1(\mathbb D) + \lambda_1^2(\mathbb D_{\rho}) A_{r,2}^2(\mathbb D)
\bigl(\|\varphi' \mid L^{\alpha}(\mathbb D)\| + \pi^{\frac{1}{\alpha}}\bigr) \|\varphi'-1 \mid L^{2}(\mathbb D)\|
\end{equation}
and
\begin{equation}\label{Estim_1}
\lambda_2(\Omega) \geq 
\lambda_2(\mathbb D) - \lambda_{*}^2 \cdot \lambda_1^2(\mathbb D_{\rho}) A_{r,2}^2(\mathbb D)
\bigl(\|\varphi' \mid L^{\alpha}(\mathbb D)\| + \pi^{\frac{1}{\alpha}}\bigr) \|\varphi'-1 \mid L^{2}(\mathbb D)\|.
\end{equation}

Now we estimate the integral in the right-hand side of these inequalities. According to Corollary~\ref{Est_Der} we obtain
\begin{multline}\label{Est_2}
\|\varphi'\,|\,L^\alpha(\mathbb D)\|=
\left(\iint\limits_{\mathbb D} |\varphi'(x,y)|^\alpha~dxdy\right)^{\frac{1}{\alpha}} \\
{} \leq \frac{C_\alpha K \pi^{\frac{2-\alpha}{2\alpha}}}{2}
\exp\left\{{\frac{K^2 \pi^2(2+ \pi^2)^2}{4\log3}}\right\}\cdot |\Omega|^{\frac{1}{2}}.
\end{multline}

Combining the inequalities~\eqref{Est_1} and~\eqref{Estim_1} consistently with the inequality~\eqref{Est_2} and given that
\[
A_{r,2}^2(\mathbb D) \leq \inf\limits_{p\in \left(\frac{4 \alpha}{3\alpha -2},2\right)} 
\left(\frac{p-1}{2-p}\right)^{\frac{2(p-1)}{p}}
\frac{\pi^{-\frac{\alpha +2}{2\alpha}} 4^{-\frac{1}{p}}}{\Gamma(2/p) \Gamma(3-2/p)}, \quad  r=\frac{4 \alpha}{\alpha -2},
\]
we get the required result.
\end{proof}

Taking into account Theorem~\ref{PPW} and Theorem~\ref{Dir_Est} we obtain lower bound in the 
Payne-P\'olya-Weinberger conjecture for quasidiscs:
\begin{theorem}
Let $\Omega \subset \mathbb R^2$ be a $K$-quasidisc of area $\pi$.
Then the ratio of the first two eigenvalues of the Dirichlet Laplacian satisfies 
\[
\frac{\lambda_2(\Omega)}{\lambda_1(\Omega)} \geq
\frac{\lambda_2(\mathbb D) - \lambda_{*}^2 \cdot \lambda_1^2(\mathbb D_{\rho}) M_{\alpha}(K) \|\varphi'-1 \mid L^{2}(\mathbb D)\|}
{\lambda_1(\mathbb{D})+\lambda_1^2(\mathbb D_{\rho}) M_{\alpha}(K) \|\varphi'-1 \mid L^{2}(\mathbb D)\|},
\]
where $\lambda_{*}=\frac{\lambda_2(\mathbb D)}{\lambda_1(\mathbb D)} \approx 2.539$, $\mathbb D_{\rho}$ is the largest disc inscribed in $\Omega$ and
\begin{multline*}
M_{\alpha}(K)= \inf\limits_{2< \alpha < \alpha^*} 
\Biggl\{
\inf\limits_{p\in \left(\frac{4 \alpha}{3\alpha -2},2\right)} 
\left(\frac{p-1}{2-p}\right)^{\frac{2(p-1)}{p}}
\frac{\pi^{-\frac{\alpha +2}{2\alpha}} 4^{-\frac{1}{p}}}{\Gamma(2/p) \Gamma(3-2/p)} \\
\times
\left(\frac{C_\alpha K \pi^{\frac{2-\alpha}{2\alpha}}}{2}
\exp\left\{{\frac{K^2 \pi^2(2+ \pi^2)^2}{4\log3}}\right\} \cdot |\Omega|^{\frac{1}{2}} +\pi^{\frac{1}{\alpha}} \right) \Biggr\}, \\
C_\alpha=\frac{10^{6}}{[(\alpha -1)(1- \nu(\alpha))]^{1/\alpha}}.
\end{multline*}
\end{theorem}

\section{Estimates of the high eigenvalues of Dirichlet-Laplacian}

In \cite{A99} were formulated open problems on eigenvalues of the Dirichlet Laplacian. 
In this section we give a partial answer on the problem of the ratio of the high eigenvalues.

On the basis of Theorem \ref{BGU} we prove the following result:

\begin{theorem}\label{Spec_Est}
Let $\Omega$ be a conformal $\alpha$-regular domain such that $\mathbb D \subseteq t \Omega$, $t>0$.
Then for any $k \in \mathbb N$ the following inequalities     
$$
\lambda_k(\mathbb D)- t^4 \lambda_k^2(\mathbb D) A_{r,2}^2(\mathbb D)V_{\alpha}^{0}(\mathbb D,\Omega) \leq \lambda_k(\Omega) \leq t^2 \lambda_k(\mathbb D)
$$
hold.
\end{theorem}

\begin{proof}
Since $\Omega$ is a conformal regular domain then by Theorem \ref{BGU} in case $\widetilde{\Omega}=\mathbb D$, we have
\[
|\lambda_k(\Omega)-\lambda_k(\mathbb D)|
\leq \max \left\{\lambda_k^2(\Omega),\lambda_k^2(\mathbb D)\right\} A_{r,2}^2(\mathbb D)V_{\alpha}^{0}(\mathbb D,\Omega).
\]
Using the definition of the absolute value we get
\begin{multline}\label{Inequaty_1}
-\max \left\{\lambda_k^2(\Omega),\lambda_k^2(\mathbb D)\right\} A_{r,2}^2(\mathbb D)V_{\alpha}^{0}(\mathbb D,\Omega)\\
\leq \lambda_k(\Omega)-\lambda_k(\mathbb D)\\
\leq \max \left\{\lambda_k^2(\Omega),\lambda_k^2(\mathbb D)\right\} A_{r,2}^2(\mathbb D)V_{\alpha}^{0}(\mathbb D,\Omega).
\end{multline}

Now we calculate maximum between $\lambda_k(\Omega)$ and $\lambda_k(\mathbb D)$ using
the property of domain monotonicity for the Dirichlet eigenvalues and the following equality \cite{GN13}
\[
\lambda_k(t \Omega)=\frac{\lambda_k(\Omega)}{t^2}.
\]
Hence we have
\[
\max \left\{\lambda_k^2(\Omega),\lambda_k^2(\mathbb D)\right\}= \max \left\{t^4 \lambda_k^2(t \Omega),\lambda_k^2(\mathbb D)\right\}= t^4 \lambda_k^2(\mathbb D).
\]
Taking into account the last equality we can rewrite the inequality \eqref{Inequaty_1} as
\begin{equation}\label{Equality_2}
-t^4 \lambda_k^2(\mathbb D) A_{r,2}^2(\mathbb D)V_{\alpha}^{0}(\mathbb D,\Omega)\\
\leq \lambda_k(\Omega)-\lambda_k(\mathbb D)\\
\leq t^4 \lambda_k^2(\mathbb D) A_{r,2}^2(\mathbb D)V_{\alpha}^{0}(\mathbb D,\Omega).
\end{equation} 

Because $\mathbb D \subseteq t \Omega$ we obtain by straightforward calculations  the following upper estimate for eigenvalues of the Dirichlet-Laplacian  
$$
\lambda_k(\Omega) = t^2 \lambda_k(t \Omega) \leq  t^2 \lambda_k(\mathbb D).
$$  
Now consider the lower estimate of \eqref{Equality_2} and the last upper estimate. So we obtain   
$$
\lambda_k(\mathbb D)- t^4 \lambda_k^2(\mathbb D) A_{r,2}^2(\mathbb D)V_{\alpha}^{0}(\mathbb D,\Omega) \leq \lambda_k(\Omega) \leq t^2 \lambda_k(\mathbb D).
$$
\end{proof}

As a consequence of Theorem \ref{Spec_Est} we obtain asymptotically exact lower estimates for ratios of Dirichlet eigenvalues in the case of conformal regular domains.
\begin{corollary}\label{Spec_Est_1}
Let $\Omega$ be a conformal $\alpha$-regular domain and $\mathbb D \subseteq t \Omega$ for $t>0$.
Then for any $m,n \in \mathbb N$, $m<n$, the following inequality     
$$
\frac{\lambda_n(\Omega)}{\lambda_m(\Omega)} \geq
\frac{\lambda_n(\mathbb D)- t^4 \lambda_n^2(\mathbb D) A_{r,2}^2(\mathbb D)V_{\alpha}^{0}(\mathbb D,\Omega)}{t^2 \lambda_m(\mathbb D)} 
$$
holds.
\end{corollary}

As an illustration we again consider the domains bounded by the epicycloid.
\begin{example}
For $k \in \mathbb{N}$, the diffeomorphism 
$$
\psi(z)=z+\frac{1}{k}z^k, \quad z=x+iy,
$$
is conformal and maps the unit disc $\mathbb D$ onto the domain $\Omega_k$ bounded by an epicycloid of $(k-1)$ cusps, inscribed in the circle $|w|=(k+1)/k$. 
Note that $\mathbb D \not\subset \Omega_k$. However, if put $t=k^2/(k-1)^2$ then $\mathbb D \subseteq t \Omega_k$.
Then by Corollary \ref{Spec_Est_1} we have    
$$
\frac{\lambda_n(\Omega_k)}{\lambda_m(\Omega_k)} \geq
\frac{\lambda_n(\mathbb D)- \frac{k^8}{(k-1)^8} \lambda_n^2(\mathbb D) A_{r,2}^2(\mathbb D)V^0_{\alpha}(\mathbb D,\Omega_k)}{\frac{k^4}{(k-1)^4} \lambda_m(\mathbb D)}.
$$
\end{example}

\textbf{Acknowledgements.} The first author was supported by the United States-Israel Binational Science Foundation (BSF Grant No. 2014055). The second author
was partly supported by the Ministry of Education and Science of the Russian Federation 
in the framework of the research Project No. 2.3208.2017/4.6, by RFBR Grant No. 18-31-00011.

\vskip 0.3cm

Department of Mathematics, Ben-Gurion University of the Negev, P.O.Box 653, Beer Sheva, 8410501, Israel 
 
\emph{E-mail address:} \email{vladimir@math.bgu.ac.il} \\           
       
 Division for Mathematics and Computer Sciences, Tomsk Polytechnic University, 634050 Tomsk, Lenin Ave. 30, Russia; International Laboratory SSP \& QF, Tomsk State University, 634050 Tomsk, Lenin Ave. 36, Russia

 \emph{Current address:} Department of Mathematics, Ben-Gurion University of the Negev, P.O.Box 653, 
  Beer Sheva, 8410501, Israel  
							
 \emph{E-mail address:} \email{vpchelintsev@vtomske.ru}   \\
			  
	Department of Mathematics, Ben-Gurion University of the Negev, P.O.Box 653, Beer Sheva, 8410501, Israel 
							
	\emph{E-mail address:} \email{ukhlov@math.bgu.ac.il}

\end{document}